\documentclass[11pt]{article}
\usepackage{graphicx,mathrsfs,amssymb}
\usepackage{amsmath,latexsym, color, amsbsy,amssymb,cancel}
\usepackage{graphicx}

\makeatletter
\newcommand{\superimpose}[2]{%
	{\ooalign{$#1\@firstoftwo#2$\cr\hfil$#1\@secondoftwo#2$\hfil\cr}}}
\makeatother

\def\ds{\displaystyle}
\def\forall{\hbox{for all}~}

\def\ve{\varepsilon}
\def\n{\noindent}

\def\R{\mathbb{R}}

\def\vs{\vskip 2em}

\def\v{\vskip 1em}

\def\bega{\begin{array}}
	\def\enda{\end{array}}
\def\begi{\begin{itemize}}
	\def\endi{\end{itemize}}

\def\bel{\begin{equation}\label}
	\def\eeq{\end{equation}}
\def\sqr#1#2{\vbox{\hrule height .#2pt
		\hbox{\vrule width .#2pt height #1pt \kern #1pt
			\vrule width .#2pt}\hrule height .#2pt }}
\def\square{\sqr74}
\def\endproof{\hphantom{MM}\hfill\llap{$\square$}\goodbreak}

\newtheorem{theorem}{Theorem}[section]
\newtheorem{corollary}[theorem]{Corollary}
\newtheorem{definition}[theorem]{Definition}
\newtheorem{remark}[theorem]{Remark}
\newtheorem{lemma}[theorem]{Lemma}
\newtheorem{proposition}{Proposition}[section]

\begin{document}
	\title{\bf Wave breaking for   perturbed Burgers equations}\vs
	\author{\it  Ethan Botelho, Khai T. Nguyen, and  Madhumita Roy \\
		\\
		{\small Department of Mathematics, North Carolina State University}\\
		{\small e-mails: embotelh@ncsu.edu,~ khai@math.ncsu.edu,~ mroy5@ncsu.edu}}
	\maketitle
	\begin{abstract}	
		We establish a simple and explicit criterion for wave breaking for a general class of perturbed Burgers equations that cover several Burgers-type models, including the Fractional KdV equation, the Whitham equation, and the Fornberg-Whitham equation. The proof is both rigorous and straightforward.
		\quad\\
		\quad\\
		{\footnotesize
			\medskip
			
			\n {\bf AMS Mathematics Subject Classification.}   76B15, 76B03, 35S30, 35A20
		}
	\end{abstract}
	
	\section{Introduction}
	\label{sec:1}
	We study a class of perturbed Burgers-type equations of the form
	\begin{equation}\label{BG}
		u_t(t,x) + \left( \frac{u^2(t,x)}{2} \right)_x = {\bf N}[u(t,\cdot)](x), 
		\qquad 
		u(0,\cdot) = \bar{u},
	\end{equation}
	where the source term \( {\bf N}: H^1(\mathbb{R}) \to {\bf L}^2(\mathbb{R}) \) is a bounded linear operator satisfying:
	
	\begin{itemize}
		\item[{\bf (A1)}] For every \( g \in \mathcal{C}^1(\mathbb{R})\cap  H^1(\mathbb{R}) \),
		\[
		{\bf N}[g]'(x) = {\bf N}[g'](x), 
		\qquad 
		\int_{\mathbb{R}} {\bf N}[g](x) g(x)\, dx = 0.
		\]
		
		\item[{\bf (A2)}] There exist constants \( \eta_0 \in (0,\infty] \), \( \alpha_i>0 \), and \( \lambda_i \ge 0 \) for \( i=1,2,3 \) such that
		\[
		\|{\bf N}[g]\|_{{\bf L}^{\infty}(\R)} 
		\le 
		\lambda_1 \eta^{-\alpha_1} \|g\|_{{\bf L}^{\infty}(\R)} 
		+ \lambda_2 \eta^{\alpha_2} \|g'\|_{{\bf L}^{2}(\R)} 
		+ \lambda_3 \eta^{\alpha_3} \|g'\|_{{\bf L}^{\infty}(\R)}
		\]
		for all \( \eta \in (0, \eta_0) \) and \( g \in H^2(\mathbb{R}) \).
	\end{itemize}
	These assumptions cover a broad family of operators, ranging from singular to weakly singular kernels. With suitable choices of $(\alpha_1, \alpha_2, \alpha_3)$, equation \eqref{BG} recovers several important nonlocal Burgers-type models where the interplay between nonlinear steepening and nonlocal effects is central such as \cite{BiH, L, BN, BZ, BNK, CE, FN, GN, G, HI, HIDT, Hu, NS, RY, W}. In particular, it includes well-known nonlocal weakly  dispersive perturbations of the Burgers equation:
	\begin{itemize}
		\item  {\bf Fractional KdV equation  \cite{Hu1, SW}}. For some $\alpha\in (-1,-2/5)$, it holds
		\[
		{\bf N}[u(t,\cdot)](x)~=~|D|^{\alpha}u_x(t,x)~=~\int_{\R}{\mathrm{sgn}(y)\over |y|^{2+\alpha}} \big[u(t,x)-u(t,x-y)\big]dy.
		\] 
		\item {\bf Whitham equation \cite{ MLQ, W}.} Considering the kernel 
		\[
		K(x)~=~\mathcal{F}^{-1}[m](x)~\doteq~{1\over 2\pi}\int_{\R}e^{ix\xi}m(\xi)d\xi,\qquad m(\xi)~=~\sqrt{{\tanh \xi\over \xi}},
		\]
		the operator ${\bf N}$ is defined by 
		\[
		{\bf N}[u(t,\cdot)]~=~K*u_x(t,\cdot).
		\]
		\item {\bf Fornberg-Whitham equation  \cite{SW1}.} For some $s>0$, the operator ${\bf N}$ is 
		the Bessel potential of order $s>0$ such that 
		\[
		{\bf N}[u(t,\cdot)]~=~-(I-\partial^2_x)^{-s/2}*u_x(t,\cdot)~=~-G_s*u_x(t,\cdot),
		\]
		where  the Bessel kernel $G_s$  has the formula
		\[
		G_s(x)~=~{1\over 2\sqrt{\pi}\Gamma(s/2)}\int_{0}^{\infty}t^{\frac{s-3}{2}}\cdot e^{-{|x|^2\over 4t}-t}dt,
		\]
		with $\Gamma$ being the Gamma  function.
	\end{itemize}
	While the classical Burgers equation forms steep gradients in finite time, the presence of a nonlocal source may slow down, speed up, or prevent this process. 
	\begin{definition}
		The Cauchy problem \eqref{BG} exhibits \emph{wave breaking} (or \emph{shock formation}) at time \( T^*>0 \) if it admits a unique classical solution \( u \) on \( [0,T^*) \times \mathbb{R} \) such that
		\[
		\limsup_{t \to T^{*-}}\left( \inf_{x \in \mathbb{R}} u_x(t,x)\right)~ =~ -\infty.
		\]
	\end{definition}
	In this paper, our aim is to provide simple, verifiable conditions under which the nonlocal term ${\bf N}$ allows solutions to form shocks. 
	In Theorem~\ref{Main}, for every sufficiently small parameter $\theta>0$, we establish explicit criteria for wave breaking in terms of the slope of the initial data $\bar{u}$ and the ${\bf L}^2$-norms of its derivatives.  These criteria yield bounds on the lifespan of classical solutions and identify regimes in which nonlinear steepening dominates the nonlocal effects. More precisely, we obtain  that the wave breaking time $T^*$ is bounded by 
	\begin{equation*}
		\frac{1}{(1+\theta)}\cdot\frac{1}{-\displaystyle\inf_{x\in\mathbb{R}} \bar{u}'(x)}
		\;\leq\;
		T^*
		\;\leq\;
		\frac{1}{(1-\theta)}\cdot\frac{1}{-\displaystyle\inf_{x\in\mathbb{R}} \bar{u}'(x)}.
	\end{equation*}
	The proof of Theorem \ref{Main} is subtle yet direct and is based on analyzing the blow-up of the quantity
	\[
	m(t) ~\doteq~ -\ds\inf_{x\in\mathbb{R}} u_x(t,x).
	\]
	By tracking $u_x$ along the characteristic curves of \eqref{BG} and differentiating the equation with respect to $x$, we derive a closed system of ODEs for 
	$m(t)$ and the ${\bf L}^2$-norms of $u_x(t,\cdot)$, $u_{xx}(t,\cdot)$, and $u_{xxx}(t,\cdot)$.
	This system captures the mechanism of gradient amplification and allows us to determine an interval in which wave breaking occurs, linking it to the initial slope and to the influence of the nonlocal term.
	
	In the remaining part of the paper, we apply our main result, Theorem~\ref{Main}, to establish wave-breaking criteria in Proposition~\ref{L1-ker} for a class of nonlocal dispersive equations of the form
	\[
	u_t + u u_x = K*u,
	\]
	where the lower-order source term consists of a spatial convolution with an odd, integrable function $K:\mathbb{R}\to\mathbb{R}$. This class includes, in particular, the Burgers-Poisson equation \cite{GN, G}. For the case of nonintegrable kernels $K$, we also derive, in Section~\ref{S3}, simpler and more explicit  wave-breaking criteria than those obtained in \cite[Theorem~2.4]{SW} for the fractional Korteweg-de Vries equation (see Proposition~\ref{fKdV}), in \cite[Theorem~2.3]{SW} for the Whitham equation (see Proposition~\ref{Whitham}), and in \cite[Theorems~2.1 and~2.2]{SW1} for the Fornberg-Whitham equation (see Proposition~\ref{FW}).
	
	\section{Main results}
	\setcounter{equation}{0}
	In this section, we study wave breaking for  a class of  perturbed Burgers equations (\ref{BG}) under the assumptions {\bf (A1)}-{\bf (A2)} on the nonlocal operator ${\bf N}$. For given constants $\alpha_i$ and $\lambda_i$ in {\bf (A2)}, we define
	\bel{al}
	\bar{\alpha}_2~\doteq~{\alpha_1\over 2\alpha_1+\alpha_2},\qquad \bar{\alpha}_3~\doteq~{2\alpha_1\over 5\alpha_1+3\alpha_3},\qquad \theta_0~\doteq~\min\left\{{2\alpha_2-\alpha_1\over 4\alpha_1+2\alpha_2},{3\alpha_3-2\alpha_1\over 5\alpha_1+3\alpha_3 }\right\},
	\eeq
	\bel{Cons}
	C_2~\doteq~3^{1-\bar{\alpha}_2}\lambda^{1-2\bar{\alpha}_2}_1\lambda_2^{\bar{\alpha}_2},\qquad C_3~\doteq~3^{1-\bar{\alpha}_3}\lambda_1^{{2-5\bar{\alpha}_3\over 2}}\big(C_{GN}\lambda_3\big)^{^{{3\bar{\alpha}_3\over 2}}},
	\eeq
	where $C_{GN}$ is the optimal constant in the Gagliardo--Nirenberg interpolation inequality:
	\begin{equation}\label{C-GN}
		C_{GN} \;\doteq\;
		\sup_{0\neq f\in H^2(\R)}
		\frac{\|\partial_x f\|_{{\bf L}^{\infty}(\mathbb{R})}}
		{ \|f\|_{{\bf L}^{\infty}(\mathbb{R})}^{1/3}
			\,\|\partial_x^{2} f\|_{{\bf L}^{2}(\mathbb{R})}^{2/3}}.
	\end{equation}
	Our main result is stated below.
	\begin{theorem}\label{Main} 
		Assume that the linear operator ${\bf N}$ satisfies {\bf (A1)}-{\bf (A2)} for $\alpha_1<\min\{2\alpha_2, 3\alpha_3/2\}$. Then for every $\theta\in (0,\theta_0] $ and $\bar{u}\in H^3(\R)$ with $\gamma_{\bar{u}}\doteq -\ds\inf_{x\in \R}\bar{u}'(x)/\|\bar{u}'\|_{{\bf L}^\infty(\R)}>0$ such that 
		\bel{cond-u0}
		-\inf_{x\in \R}\bar{u}'(x)~>~\max\left\{{3 \lambda_1\over \theta\gamma_{\bar{u}}\eta_0^{\alpha_1}},{C_2\over \theta^{1-\bar{\alpha}_2}\gamma_{\bar{u}}^{1-2\bar{\alpha}_2}}\cdot \|\bar{u}''\|^{\bar{\alpha}_2}_{{\bf L}^2(\R)}+{C_3\over \theta^{1-\bar{\alpha}_3}\gamma_{\bar{u}}^{1-2\bar{\alpha}_3}}\cdot \|\bar{u}'''\|^{\bar{\alpha}_3}_{{\bf L}^2(\R)}\right\},
		\eeq
		the Cauchy problem (\ref{BG}) exhibits wave breaking at some time $T^*>0$ such that 
		\bel{b-T*}
		{1\over (1+\theta)}\cdot {1\over -\ds\inf_{x\in \R}\bar{u}'(x)}\,~\leq~T^*~\leq~{1\over (1-\theta)}\cdot {1\over -\ds\inf_{x\in \R}\bar{u}'(x)}\,.
		\eeq
	\end{theorem}
	{\bf Proof.} {\bf 1.} We first note that for a given  $\bar{u}\in H^3(\R)$, the above Cauchy problem (\ref{BG}) admits  a unique classical solution $u\in \mathcal{C}([0,t),H^3(\R))$ for some $t>0$ from \cite{Kato}. Hence, we denote
	\bel{T^*}
	T^*~=~\sup\{t\geq 0:(\ref{BG})~\mathrm{admits~a~unique~solution~}u\in \mathcal{C}([0,t],H^{3}(\R))\}~>~0.
	\eeq
	Using~(\ref{BG}) and~{\bf (A1)}, we obtain
	\[
	\frac{d}{dt}\int_{\mathbb{R}} u^{2}(t,x)\, dx
	~=~ 2
	\int_{\mathbb{R}} {\bf N}[u(t,\cdot)](x)\, u(t,x)\, dx~=~0,
	\]
	and this yields
	\begin{equation}\label{L-2}
		\|u(t,\cdot)\|_{{\bf L}^{2}(\mathbb{R})}^{2}
		~=~
		\|\bar{u}\|_{{\bf L}^{2}(\mathbb{R})}^{2},
		\qquad t\in [0,T^*].
	\end{equation}
	Differentiating \eqref{BG} with respect to $x$ and multiplying by $u_x$, we obtain
	\bel{ux}
	(u_x)_t+u u_{xx}+u^2_x~=~ {\bf N}[u_{x}],\qquad {d\over dt}\big(u^2_x\big)+ \big(u u^2_x\big)_x+ u_x^{3}~=~2 {\bf N}[u_{x}] u_x,
	\eeq
	which implies
	\bel{u-x-L2}
	{d\over dt}\int_{\R} u^2_xdx~=~-\int_{\R} \left( u^3_x - 2 {\bf N}[u_x] u_x \right) \,dx~=~-\int_{\R} u^3_x \, dx.
	\eeq
	Similarly, differentiating \eqref{BG} twice (respectively, three times) with respect to $x$ and multiplying by $u_{xx}$ (respectively, $u_{xxx}$), we obtain
	\begin{equation}\label{u-xx-L2}
		\frac{d}{dt}\int_{\mathbb{R}} u_{xx}^2\,dx ~=~-5\int_{\mathbb{R}} u_x\,u_{xx}^2\,dx,
		\qquad
		\frac{d}{dt}\int_{\mathbb{R}} u_{xxx}^2\,dx ~=~ -7\int_{\mathbb{R}} u_x\,u_{xxx}^2\,dx.
	\end{equation}
	For convenience, for every $t\in [0,T^*)$, we introduce the quantities
	\bel{m}
	m(t)~\doteq~-\inf_{x\in \R} u_x(t,x),\qquad M(t)~\doteq~\|u_x(t,\cdot)\|_{{\bf L}^{\infty}(\R)},
	\eeq
	and 
	\bel{z12}
	z_0(t)~\doteq~\|u(t,\cdot)\|_{{\bf L}^2(\R)}^{{2}},\qquad z_1(t)~\doteq~\|u_x(t,\cdot)\|_{{\bf L}^2(\R)}^{{2}}.
	\eeq
	From (\ref{L-2})-(\ref{u-x-L2}), it holds
	\bel{z2}
	z_0(t)~=~\|\bar{u}\|^2_{{\bf L}^2(\R)},\qquad \dot{z}_1(t)~\leq~m(t) z_1(t).
	\eeq
	Moreover, let $(z_2,z_3)$ be the solution of
	\bel{z3}
	\begin{cases}
		\dot{z}_2(t)~=~5m(t)z_2(t),\qquad z_2(0)~=~\| \bar{u}'' \|^2_{{\bf L}^2(\R)},\\[2mm]
		\dot{z}_3(t)~=~7m(t)z_3(t),\qquad z_3(0)~=~\| \bar{u}''' \|^2_{{\bf L}^2(\R)}.
	\end{cases}
	\eeq
	By a standard comparison principle, we have 
	\[
	\|u_{x x}(t,\cdot)\|^2_{{\bf L}^2(\R)}~\leq~z_2(t),\qquad \|u_{x x x}(t,\cdot)\|^2_{{\bf L}^2(\R)}~\leq~z_3(t).
	\]
	Hence, as long as $m$ remains bounded, the $z_i$'s stays bounded for $i=0,1,2,3$. Consequently, $\|u(t,\cdot)\|_{H^3(\mathbb{R})}$ also remains bounded, which implies that   $T^*$ defined in (\ref{T^*}) corresponds to the wave breaking time.
	\v
	
	{\bf 2.} For every $\beta\in \R$, let $t\mapsto \xi(t;\beta)$ be the characteristic starting from $\beta$ at time $t=0$ which solves the ODE
	$$
	\dot{\xi}(t;\beta)~=~ u(t,\xi(t;\beta)),\qquad \xi(0;\beta)~=~\beta.
	$$
	The map $t\mapsto v(t;\beta)\doteq -u_x(t,\xi(t;\beta))$ is thus the solution of 
	\bel{v-b}
	\begin{cases}
		\dot{v}(t;\beta)~=~v^2(t;\beta)-{\bf N}[u_x(t,\cdot)](\xi(t;\beta)),\qquad t\in (0,T^*),\\[2mm]
		v(0;\beta)~=~-\bar{u}'(\beta),
	\end{cases}
	\eeq
	and this satisfies 
	\bel{m-v}
	\sup_{\beta\in \R}v(t;\beta)~=~m(t),\qquad \sup_{\beta\in \R}v^2(t ;\beta)~=~M^2(t)\qquad\forall t\in [0,T^*).
	\eeq
	By the assumption {\bf (A2)} and (\ref{C-GN}), we get 
	\bel{k-ex}
	\begin{split}
		\big|\dot{v}(t;\beta)-v^2(t;\beta)\big|&~\leq~\lambda_1\eta^{-\alpha_1}M(t)+\lambda_2\eta^{\alpha_2}z^{1/2}_2(t)+\lambda_3\eta^{\alpha_3}\|\partial^2_xu(t,\cdot)\|_{{\bf L}^\infty(\R)}\\[2mm]
		&~\leq~\lambda_1\eta^{-\alpha_1}M(t)+\lambda_2\eta^{\alpha_2}z^{1/2}_2(t)+C_{GN}\lambda_3 \eta^{\alpha_3} M^{1/3}(t)z_3^{1/3}(t).
	\end{split}
	\eeq
	For every $0<\ve<\gamma_{\bar{u}}$, we set 
	\bel{gamm-ve}
	0~<~\gamma_{\ve}~\doteq~ \gamma_{\bar{u}}-\ve~<~{m(0)\over M(0)}.
	\eeq
	From (\ref{cond-u0}) and the Lipschitz continuity  of $m$ and $M$, for $\ve>0$ sufficiently small, it holds
	\bel{T-1}
	T_{1,\ve}~\doteq~\sup\left\{t\in [0,T^*): {m(\tau)\over M(\tau)}> \gamma_{\ve},\left({\theta\gamma_{\ve}m(\tau)\over 3\lambda_1}\right)^{-1/\alpha_1} <\eta_0~~\forall \tau\in [0,t] \right\}~>~0.
	\eeq
	In particular, for every $t\in [0,T_{1,\ve})$, by choosing 
	\bel{eta}
	\eta~=~\left({\theta\gamma_{\ve}\over 3\lambda_1}\right)^{-1/\alpha_1}m^{-1/\alpha_1}(t),
	\eeq
	we derive from (\ref{k-ex}) that 
	\bel{k-ex1}
	\begin{split}
		\big|\dot{v}(t;\beta)-v^2(t;\beta)\big|&~\leq~{\theta \gamma_{\ve}M(t)\over 3} m(t)+~\lambda_2\left({3\lambda_1\over \theta\gamma_{\ve}}\right)^{\alpha_2\over\alpha_1} { z_2^{{1\over 2}}(t)\over m^{{\alpha_2\over\alpha_1}}(t)}
		+C_{GN}\lambda_3\left({3\lambda_1\over \theta\gamma_{\ve}}\right)^{\alpha_3\over\alpha_1}{z_3^{{1\over 3}}(t)M^{{1\over 3}}(t)\over m^{{\alpha_3\over\alpha_1}}(t)}\\[2mm]
		&~\leq~{\theta m^{2}(t) \over 3} + {\Lambda_{2,\ve} z_2^{1/2}(t)\over \left[m(t)\right]^{\alpha_2/\alpha_1}}
		+{\Lambda_{3,\ve}z_3^{1/3}(t) \over \left[m(t)\right]^{\alpha_3/\alpha_1-1/3}}
	\end{split}
	\eeq
	with the constants 
	\bel{const}
	\Lambda_{2,\ve}~\doteq~\lambda_2\left({3\lambda_1\over \theta\gamma_{\ve}}\right)^{\alpha_2\over\alpha_1}\qquad\mathrm{and}\qquad \Lambda_{3,\ve}~\doteq~C_{GN}\lambda_3\left({3\lambda_1\over \theta}\right)^{\alpha_3\over\alpha_1}\left({1\over \gamma_{\ve}}\right)^{{\alpha_3\over \alpha_1}+{1\over 3}}.
	\eeq
	In the following steps, we shall use the ODEs (\ref{z3}) and (\ref{k-ex1}) to show that $m$ goes to $+\infty$ in finite time, and this yields the wave breaking for (\ref{BG}).
	\medskip
	
	{\bf 3.} Recalling the constants $\bar{\alpha}_2$ and $\bar{\alpha}_3$ defined in (\ref{al}), we introduce the increasing map $[0,T^*]\ni t\mapsto Z_{\ve}(t)$ defined by 
	\bel{Z}
	Z_{\ve}(t)~\doteq~\left({3\Lambda_{2,\ve} \over\theta }\right)^{\bar{\alpha}_2}\cdot z^{{\bar{\alpha}_2\over 2}}_2(t)+\left({3\Lambda_{3,\ve} \over\theta }\right)^{3{\bar{\alpha}_3}\over 2}\cdot z_3^{{\bar{\alpha}_3}\over 2}(t),
	\eeq
	such that
	\bel{con-Z}
	z_2^{1/2}(t)~\leq~{\theta\over 3\Lambda_{2,\ve}}\cdot Z_{\ve}^{{1\over \bar{\alpha}_2}}(t),\qquad z_3^{1/3}(t)~\leq~{\theta\over 3\Lambda_{3,\ve}}\cdot Z_{\ve}^{{2\over 3\bar{\alpha}_3}}(t).
	\eeq
	Using (\ref{al})-(\ref{C-GN}), (\ref{z3}),  (\ref{const})-(\ref{Z}), and $\theta<\theta_0$, we compute 
	\bel{Z'}
	\begin{split}
		\dot{Z}_{\ve}(t)&~=~{5m(t)\bar{\alpha}_2\over 2}\left({3\Lambda_{2,\ve} \over\theta }\right)^{\bar{\alpha}_2}\cdot z^{{\bar{\alpha}_2\over 2}}_2(t)+{7m(t)\bar{\alpha}_3\over 2}\left({3\Lambda_{3,\ve} \over\theta }\right)^{3{\bar{\alpha}_3}\over 2}\cdot z_3^{{\bar{\alpha}_3}\over 2}(t)\\[2mm]
		&~\leq~\max\left\{{5\bar{\alpha}_2\over 2}, {7\bar{\alpha}_3\over 2}\right\}\cdot m(t)Z_{\ve}(t)~\leq~(1-\theta_0)m(t)Z_{\ve}(t),\qquad t\in (0,T^*),
	\end{split}
	\eeq
	and
	\[
	\begin{split}
		Z_{\ve}(0)&~=~\left({3\Lambda_{2,\ve} \over\theta }\right)^{\bar{\alpha}_2} \cdot \|\partial_x^2 \bar{u} \|^{\bar{\alpha}_2}_{{\bf L}^2(\R)}+\left({3\Lambda_{3,\ve} \over\theta }\right)^{3{\bar{\alpha}_3}\over 2}\cdot \|{\partial_x^3 \bar{u}} \|^{\bar{\alpha}_3}_{{\bf L}^2(\R)}\\[2mm]
		&~=~{C_2\over \theta^{1-\bar{\alpha}_2} \gamma^{1-2\bar{\alpha}_2}_{\ve}} \cdot \| \bar{u}'' \|^{\bar{\alpha}_2}_{{\bf L}^2(\R)} + {C_3\over \theta^{1-\bar{\alpha}_3} \gamma^{1-2\bar{\alpha}_3}_{\ve}} \cdot \|\bar{u}''' \|^{\bar{\alpha}_3}_{{\bf L}^2(\R)}.
	\end{split}
	\]
	By the assumption (\ref{cond-u0}) and (\ref{gamm-ve}), for $\ve>0$ sufficiently small, we have 
	\[
	m(0)~=~-\inf_{x\in \R}\bar{u}'(x)~>~\left({3\Lambda_{2,\ve} \over\theta }\right)^{\bar{\alpha}_2}\cdot z^{{\bar{\alpha}_2\over 2}}_2(0)+\left({3\Lambda_{3,\ve} \over\theta }\right)^{3{\bar{\alpha}_3}\over 2}\cdot z_3^{{\bar{\alpha}_3}\over 2}(0)~=~Z_{\ve}(0),
	\]
	and the Lipschitz continuity of $m$ implies that 
	\bel{T2}
	T_{2,\ve}~\doteq~\sup\left\{t\in [0,T_{1,\ve}):m(\tau)\geq Z_{\ve}(\tau)~\forall \tau\in [0,t]\right\}~>~0.
	\eeq
	Hence, from (\ref{k-ex1}) and  (\ref{con-Z}), for every $t\in (0,T_{2,\ve})$, it holds
	\bel{k-ex2}
	\begin{split}
		\big|\dot{v}(t;\beta)-v^2(t;\beta)\big|&~\leq~{\theta\over 3} m^2(t)+{\theta\over 3} \left[m(t)\right]^{-{\alpha_2\over \alpha_1}} Z_{\ve}^{{1\over \bar{\alpha}_2}}(t) + {\theta\over 3} \left[m(t)\right]^{{1\over 3}-{\alpha_3\over \alpha_1}} Z_{\ve}^{{2\over 3\bar{\alpha}_3}}(t)\\[2mm]
		&~\leq~{\theta\over 3} m^2(t)+{\theta\over 3} \left[m(t)\right]^{{1\over \bar{\alpha}_2}-{\alpha_2\over \alpha_1}}+{\theta\over 3} \left[m(t)\right]^{{2\over 3\bar{\alpha}_3}+{1\over 3}-{\alpha_3\over \alpha_1}}~=~\theta m^{2}(t).
	\end{split}
	\eeq
	Using (\ref{Z'})-(\ref{k-ex2}), we get
	\[
	\begin{split}
		{d\over dt} \left(v(t;\beta)-Z_{\ve}(t)\right)&~\geq~v^2(t;\beta)-\theta m^2(t)-\dot{Z}_{\ve}(t)~\geq~v^2(t;\beta)-\theta m^2(t)-(1-\theta_0)m(t)Z_{\ve}(t)\\
		&~\geq~v^2(t;\beta)-\theta m^2(t)-(1-\theta_0)m^2(t)~\geq~v^2(t;\beta)-m^2(t).
	\end{split}
	\]
	In particular, for every $\beta\in \R$, $t\in [0,T_{2,\ve})$, and $0\leq s<T_{2,\ve}-t$, one has 
	\[
	\begin{split}
		m(t+s)-Z_{\ve}(t+s)&~\geq~v(t+s;\beta)-Z_{\ve}(t+s)\\
		&~\geq~v(t;\beta)-Z_{\ve}(t)+\int_{t}^{t+s} \left( v^2{(\tau;\beta)}-m^2(\tau) \right)d\tau\\
		&~\geq~v(t;\beta)-Z_{\ve}(t)+ s\cdot \big(v^2(t;\beta)-m^2(t)\big)+o(s),
	\end{split}
	\]
	and (\ref{m-v}) implies 
	\bel{m_Z}
	m(t+s)-Z_{\ve}(t+s)~\geq~m(t)-Z_{\ve}(t)+o(s).
	\eeq
	By the Lipschitz continuity property of the map $t\mapsto m(t)$, we get
	\[
	{d\over dt} \big[m(t)-Z_{\ve}(t)\big]~\geq~0\quad a.e.~t\in [0,T_{2,\ve}],
	\]
	which in particular yields that $t\mapsto m(t)$ is increasing and
	\bel{Ne1}
	m(t)-Z_{\ve}(t)~\geq~ m(0)-Z_{\ve}(0)~>~0\quad\forall t\in [0,T_{2,\ve}].
	\eeq
	Hence, by the definition of $T_{2,\varepsilon}$ in \eqref{T2}, we have $T_{2,\varepsilon} = T_{1,\varepsilon}$, and
	\bel{ve-e}
	\left(\frac{\theta \gamma_{\varepsilon} m(T_{1,\varepsilon})}{3\lambda_1}\right)^{-1/\alpha_1}
	~\le~
	\left(\frac{\theta \gamma_{\varepsilon} m(0)}{3\lambda_1}\right)^{-1/\alpha_1}
	< \eta_0.
	\eeq
	Moreover, from (\ref{k-ex2}), we get
	\begin{equation}\label{kes-1}
		v^2(t;\beta) - \theta m^2(t)
		~\le~
		\dot{v}(t;\beta)
		~\le~
		v^2(t;\beta) + \theta m^2(t),
	\end{equation}
	for all $t \in [0, T_{1,\varepsilon}]$.\medskip

	\noindent\textbf{4.} Next, we show that
	\begin{equation}\label{ce}
		\gamma_{\varepsilon} M(t) ~\le~ m(t)
		\qquad \forall\, t \in [0, T_{1,\varepsilon}].
	\end{equation}
	Arguing by contradiction, assume that
	\begin{equation}\label{t-0}
		t_0 ~\doteq~ \sup \Big\{ t \in [0, T_{1,\varepsilon}] : \gamma_{\varepsilon} M(\tau) \le m(\tau)
		~~\forall\, \tau \in [0,t] \Big\}
		~<~ T_{1,\varepsilon}.
	\end{equation}
	Then
	\begin{equation}\label{dep}
		\sup_{\beta \in \mathbb{R}} \big[ -v(t_0;\beta) \big]
		~=~ M(t_0)
		~=~ \frac{m(t_0)}{\gamma_{\varepsilon}}
		~>~ m(t_0).
	\end{equation}
	Using the first inequality in \eqref{kes-1}, for every $0 \leq s < T_{1, \ve} - t_{0}$, we estimate
	\[
	\begin{aligned}
		M(t_0)
		&\geq~ -v(t_0;\beta)~\ge~ -v(t_0+s;\beta)
		+ \int_{t_0}^{t_0+s} \big( v^2(\tau;\beta) - \theta m^2(\tau) \big)\, d\tau \\[2mm]
		&~\ge~ -v(t_0+s;\beta)+s \big( v^2(t_0+s;\beta) - \theta m^2(t_0+s) \big)
		+ o(s)
	\end{aligned}
	\]
	with $\ds\lim_{s\to 0}o(s)/s=0$. By \eqref{m-v}, it follows that
	\[
	\begin{aligned}
		M(t_0)
		&~\ge~ M(t_0+s)
		+ s \big( M^2(t_0+s) - \theta m^2(t_0+s) \big)
		+ o(s) \\[2mm]
		&~\ge~ M(t_0+s)
		+  (1-\theta)s\, M^2(t_0+s)
		+ o(s).
	\end{aligned}
	\]
	In particular, for some sufficiently small $s_0>0$,  by the increasing property of $m$ proved in the previous step and \eqref{dep}, it holds that
	\[
	\gamma_{\varepsilon} M(t_0+s)
	~\le~ \gamma_{\varepsilon} M(t_0)
	~=~ m(t_0)
	~\le~ m(t_0+s)
	\quad \forall\, s \in [0, s_0),
	\]
	which contradicts \eqref{t-0}. Hence, \eqref{ce} is valid and, together with \eqref{ve-e} and \eqref{T-1}, implies that $T_{1,\varepsilon} = T^*$ and that \eqref{kes-1} holds for a.e. $t \in [0, T^*)$.
	\medskip
	
	\noindent\textbf{5.} Using the same argument as in the previous step, we deduce from the first inequality in \eqref{kes-1} that
	\[
	\dot{m}(t) \ge (1-\theta)\, m^2(t) \quad \text{for a.e.\ } t \in [0,T^*),
	\]
	which implies
	\[
	m(t) \ge \frac{m(0)}{1-(1-\theta)t\, m(0)} \qquad \forall\, t \in [0,T^*).
	\]
	Hence, the upper bound for $T^*$ in \eqref{b-T*} follows.
	
	Finally, using the second inequality in \eqref{kes-1}, for every $0 < t < t+s < T^*$, we obtain
	\[
	\begin{aligned}
		v(t;\beta)
		&~\le~ m(t-s) + \int_{t-s}^{t} \big( v^2(\tau;\beta) + \theta m^2(\tau) \big)\, d\tau \\[2mm]
		&~\le~ m(t-s) + s \big( v^2(t;\beta) + \theta m^2(t) \big) + o(s)\qquad\forall \beta\in \R.
	\end{aligned}
	\]
	Together with \eqref{m-v}, this yields
	\[
	\dot{m}(t)~ \le~ (1+\theta)\, m^{2}(t) \quad \text{for a.e.\ } t \in [0,T^*).
	\]
	Therefore, we get 
	\begin{equation}\label{low-b}
		m(t)~\le~ \frac{m(0)}{1-(1+\theta)t\, m(0)} \qquad \forall t \in [0,T^*),
	\end{equation}
	which gives the lower bound for $T^*$ in \eqref{b-T*}.
	\endproof
	\begin{remark}\label{R1} If the constant $\eta_0$ in {\bf (A2)} is $+\infty$, then the wave breaking condition (\ref{cond-u0}) in Theorem \ref{Main} can be simplified to 
		\[
		-\inf_{x\in \R}\bar{u}'(x)~>~{C_2\over \theta^{1-\bar{\alpha}_2}\gamma_{\bar{u}}^{1-2\bar{\alpha}_2}}\cdot \|\bar{u}''\|^{\bar{\alpha}_2}_{{\bf L}^2(\R)}+{C_3\over \theta^{1-\bar{\alpha}_3}\gamma_{\bar{u}}^{1-2\bar{\alpha}_3}}\cdot \|\bar{u}'''\|^{\bar{\alpha}_3}_{{\bf L}^2(\R)}.
		\]
	\end{remark}
	\begin{remark}
		If the constant $\lambda_3$ in the assumption {\bf (A2)} is $0$, then (\ref{b-T*}) still holds for every initial data $\bar{u}\in H^2(\R)$ satisfying
		\[
		-\inf_{x\in \R}\bar{u}'(x)~>~\max\left\{{3 \lambda_1\over \theta\gamma_{\bar{u}}\eta_0^{\alpha_1}},{C_2\over \theta^{1-\bar{\alpha}_2}\gamma_{\bar{u}}^{1-2\bar{\alpha}_2}}\cdot \|\bar{u}''\|^{\bar{\alpha}_2}_{{\bf L}^2(\R)}\right\}.
		\]
	\end{remark}
	To complete this section, we observe that if  the constant $\alpha_1$ takes the value $0$, then assumption {\bf (A2)} becomes 
	\bel{L1-Kernel}
	\|{\bf N}[g]\|_{{\bf L}^{\infty}(\R)}~\leq~\lambda_1 \|g\|_{{\bf L}^{\infty}(\R)}\qquad\forall g \in H^2(\mathbb{R}).
	\eeq
	Hence, using  Theorem~\ref{Main}, we obtain the following Proposition.
	\begin{proposition}\label{L1-ker} Assume that the linear operator ${\bf N}$ satisfies {\bf (A1)} and (\ref{L1-Kernel}). Then for every $\theta\in (0,1)$ and $\bar{u}\in H^2(\R)$ with 
		\bel{s1}
		-\inf_{x\in \R}\bar{u}'(x)~>~\sqrt{3 \lambda_1\over \theta}\cdot \| \bar{u} '\|_{{\bf L}^\infty(\R)}^{1\over 2},
		\eeq
		the Cauchy problem (\ref{BG}) exhibits wave breaking at some time $T^*>0$ satisfying (\ref{b-T*}).
	\end{proposition}
	{\bf Proof.} Under the assumption that (\ref{L1-Kernel}) holds, for every $\eta \in (0,1)$ and $\varepsilon > 0$, we have
	\[
	\|{\bf N}[g]\|_{{\bf L}^{\infty}(\mathbb{R})}
	\leq
	\lambda_1 \eta^{-\varepsilon}\,\|g\|_{{\bf L}^{\infty}(\mathbb{R})}
	\quad \text{for all } g \in H^2(\mathbb{R}).
	\]
	Hence, applying Theorem~\ref{Main} with $\alpha_1=\varepsilon$, $\alpha_{2} > 0$ and $\alpha_{3} > 0$ free, $\lambda_2=\lambda_3=0$, and $\eta_{0}=1$, we get that for every
	\[
	\theta \in \left(0,\theta_\ve\doteq \min\left\{\frac{2 \alpha_{2}-\varepsilon}{2 \alpha_{2}+4\varepsilon},\,\frac{3 \alpha_{3}-2\varepsilon}{3 \alpha_{3}+5\varepsilon}\right\}\right),
	\]
	and every $\bar{u} \in H^2(\mathbb{R})$ satisfying
	\[
	-\inf_{x\in \mathbb{R}} \bar{u}'(x)
	~>~
	\frac{3\lambda_1}{\theta \gamma_{\bar{u}}},
	\]
	the wave breaking time $T^*$ satisfies (\ref{b-T*}). Since $\displaystyle \lim_{\varepsilon\to 0}\theta_{\varepsilon}=1$ and the above inequality is equivalent to (\ref{s1}), letting $\varepsilon \to 0$ completes the proof.
	\endproof
	\medskip
	
	As a consequence of Proposition~\ref{L1-ker}, we obtain wave-breaking criteria for a class of nonlocal dispersive equations of the form
	\begin{equation}\label{KK}
		u_t(t, x) + \left( \frac{u^2(t, x)}{2} \right)_x = (K*u)(t, x),
		\qquad 
		u(0,\cdot) = \bar{u},
	\end{equation}
	where the lower-order source term consists of a spatial convolution with an odd, integrable function $K:\mathbb{R}\to\mathbb{R}$. Indeed, in this case, we have
	\[
	\|K*g\|_{{\bf L}^{\infty}(\mathbb{R})}
	\leq
	\|K\|_{{\bf L}^{1}(\mathbb{R})}\,\|g\|_{{\bf L}^{\infty}(\mathbb{R})}
	\quad \text{for all } g \in {\bf L}^{\infty}(\mathbb{R}).
	\]\begin{corollary}\label{L-ker} Assume that $K:\R\to\R$ is odd and integrable. Then for every $\theta\in (0,1)$ and $\bar{u}\in H^2(\R)$ with 
		\[
		-\inf_{x\in \R}\bar{u}'(x)~>~\sqrt{3 \|K\|_{{\bf L}^{1}(\R)}\over \theta}\cdot \|\bar{u}'\|_{{\bf L}^\infty(\R)}^{1\over 2},
		\]
		the Cauchy problem (\ref{KK}) exhibits wave breaking at some time $T^*>0$ satisfying (\ref{b-T*}).
	\end{corollary}
	
	\section{Applications to nonlocal  weakly dispersive  perturbations of the Burgers equation} \label{S3}
	\setcounter{equation}{0}
	In this section, we apply our main result, Theorem~\ref{Main}, to investigate the wave-breaking phenomenon for a class of nonlocal weakly dispersive perturbations of the Burgers equation of the form (\ref{KK}), where the kernel $K:\mathbb{R}\to\mathbb{R}$ is not integrable. In particular, we establish simpler criteria for wave breaking in several representative models, including the Fractional Korteweg-de Vries equation \cite{Liu}, the Whitham equation \cite{W}, and the Fornberg-Whitham equation \cite{SW1}.
	\medskip
	
	\n {\bf 1. Wave breaking for Fractional KdV equation.} For every $\alpha\in (-1,0)$,  consider the weakly dispersive Fractional KdV equation
	\begin{equation}\label{KDV}
		u_t(t,x) + \left( \frac{u^2(t,x)}{2} \right)_x = |D|^{\alpha}u_x(t,x), 
		\qquad 
		u(0,\cdot) = \bar{u},
	\end{equation}
	with $|D|^{\alpha}$ being the usual Fourier multiplier operator with symbol $|\xi|^{\alpha}$ such that 
	\[
	|D|^{\alpha}u_x(t,x)~=~\int_{\R}{\mathrm{sgn}(y)\over |y|^{2+\alpha}} \big[u(t,x)-u(t,x-y)\big]dy.
	\]
	Introducing the constant
	\bel{C}
	C(\alpha)~\doteq~\sqrt{3} \left( \frac{1 + \alpha}{4} \right)^{\frac{\alpha}{2}} \left( \frac{2C_{GN}}{|\alpha|} \right)^{\frac{1 + \alpha}{2}},
	\eeq
	the wave breaking criterion for \eqref{KDV} can be stated as follows:
	\begin{proposition}\label{fKdV} Assume that $\alpha\in (-1,-2/5)$. Then, for every $\theta\in \big(0,\frac{-2-5\alpha}{5+2\alpha}\big) $ and $\bar{u}\in H^3(\R)$ such that  
		\bel{cond-u0-1}
		-\inf_{x\in \R}\bar{u}'(x)~>~{C(\alpha)\over \theta^{\frac{1}{2}}}\cdot \|\bar{u}'\|_{{\bf L}^{\infty}(\R)}^{ \frac{1}{6}-{\alpha\over 3} }\cdot \|\bar{u}'''\|^{ \frac{1}{3}+\frac{\alpha}{3} }_{{\bf L}^{2}(\R)},
		\eeq
		the Fractional KdV equation (\ref{KDV}) exhibits wave breaking at  time $T^*>0$ satisfying (\ref{b-T*}).
	\end{proposition}
	{\bf Proof.} Consider the operator ${\bf N}: H^1(\mathbb{R}) \to {\bf L}^2(\mathbb{R})$ defined by
	\[
	{\bf N}[g](x)~\doteq~ |D|^\alpha g'(x)~=~\int_{\R}{\mathrm{sgn}(y)\over |y|^{2+\alpha}} \big[g(x)-g(x-y)\big]dy,\qquad g\in H^2(\R).
	\]
	It is straightforward to verify that {\bf N} satisfies assumption {\bf (A1)}. For every $\eta\in (0,\infty)$ and $x\in \R$, we estimate
	\[
	\begin{split}
		|{\bf N}[g](x)|&~\leq~ \left( \int_{|y|\leq\eta} + \int_{|y|>\eta} \right)\frac{\big| g( x) - g(x - y) \big| }{ |y|^{2 + \alpha} } ~d y\\[2mm]
		&~\leq~ \|g'\|_{{\bf L}^{\infty}(\R)}\cdot \left(\int_{|y|\leq\eta}{dy\over |y|^{1+\alpha}}\right)+2\|g\|_{{\bf L}^{\infty}(\R)}\cdot\left(\int_{|y| > \eta}{dy\over |y|^{2+\alpha}}\right)\\
		&~\leq~\frac{2}{ |\alpha|} \eta^{- \alpha} \| g' \|_{{\bf L}^{\infty}(\R)}+\frac{4}{1 + \alpha} \eta^{-(1 + \alpha)} \| g \|_{{\bf L}^{\infty}(\R)}.
	\end{split}
	\]
	In particular,  ${\bf N}$ satisfies assumption {\bf (A2)} with $\lambda_2=0$ and parameters
	\bel{par1}
	(\lambda_1,\lambda_3)~=~\left(\frac{4}{1 + \alpha},\frac{2}{ |\alpha|}\right),\quad (\alpha_1,\alpha_3)~=~(1+\alpha,-\alpha)\quad\mathrm{and}\quad \eta_0~=~+\infty.
	\eeq
	Since $\alpha\in (-1,-2/5)$, we have
	\[
	0~<~\alpha_1~=~1+\alpha~<~-\frac{3\alpha}{2}~=~\frac{3\alpha_3}{2},
	\]
	so that the hypothesis of Theorem~\ref{Main} is fulfilled. Hence, applying Theorem~\ref{Main} together with Remark~\ref{R1} for
	\bel{par2}
	\theta_0=\frac{-2-5\alpha}{5+2\alpha},\qquad \bar{\alpha}_3~=~ \frac{2 + 2 \alpha}{5 + 2 \alpha} ,\qquad C_2~=~0,\qquad C_3~=~3^{1-\bar{\alpha}_3}\lambda_1^{{2-5\bar{\alpha}_3\over 2}}\big(C_{GN}\lambda_3\big)^{^{{3\bar{\alpha}_3\over 2}}},
	\eeq
	we obtain the following wave breaking criterion for \eqref{KDV}
	\[
	-\inf_{x\in \R}\bar{u}'(x)~>~ {C_3\over \theta^{1-\bar{\alpha}_3}\gamma_{\bar{u}}^{1-2\bar{\alpha}_3}}\cdot \|\bar{u}'''\|^{\bar{\alpha}_3}_{{\bf L}^2(\R)}.
	\]
	Using \eqref{par1}--\eqref{par2}, \eqref{C} and  $\gamma_{\bar{u}}=-\ds\inf_{x\in \R}\bar{u}'(x)/\|\bar{u}'\|_{{\bf L}^\infty(\R)}$, the above criterion can be rewritten as \eqref{cond-u0-1}. This completes the proof.
	\endproof
	\v
	
	{\bf 2. Wave breaking for Whitham equation.} Consider the Whitham equation
	\begin{equation}\label{Whithameq}
		u_t(t,x) + \left( \frac{u^2(t,x)}{2} \right)_x
		~=~ - (K * u_x)(t,x), 
		\qquad 
		u(0, \cdot)~ =~ \bar{u},
	\end{equation}
	where the kernel $K$ is defined by
	\[
	K(x) ~=~ \mathcal{F}^{-1}[m](x)
	\doteq \frac{1}{2\pi}\int_{\mathbb{R}} e^{ix\xi} m(\xi)\,d\xi,
	\qquad 
	m(\xi)~ =~ \sqrt{\frac{\tanh \xi}{\xi}}.
	\]
	\begin{lemma}\label{B-K}The kernel \(K\) is positive, even, and decreasing on \((0,\infty)\), and satisfies
		\bel{estK}
		\int_{x}^{\infty} |K'(y)|\,dy~=~ \left| K(x) \right|~\le~\frac{1}{\sqrt{2\pi x}}, \qquad \forall x>0.
		\eeq
	\end{lemma}
	{\bf Proof.} Since the function $\xi \mapsto m(\xi)$ is positive, even, and decreasing on $(0,\infty)$, its inverse Fourier transform
	$$
	K(x)~=~\frac{1}{2\pi}\int_{\mathbb{R}} e^{ix\xi} m(\xi)~=~{1\over \pi} \int_{0}^{\infty}\cos(x\xi)m(\xi)d\xi
	$$
	is also positive, even, and decreasing on $(0,\infty)$. Indeed, for every $x>0$, we compute
	\[
	\begin{split}
		K(x)&~=~{1\over \pi} \int_{0}^{\infty}{\cos(x\xi)\over \sqrt{\xi}}d\xi-{1\over \pi} \int_{0}^{\infty}\cos(x\xi)\cdot\left({1\over \sqrt{\xi}}-m(\xi)\right)d\xi\\[2mm]
		&~=~\frac{1}{\sqrt{2\pi x}}-{1\over \pi} \int_{0}^{\infty}\cos(x\xi)\cdot{1-\sqrt{\tanh(\xi)}\over \sqrt{\xi}}d\xi.
	\end{split}
	\]
	Note that the map $\xi\mapsto \ds{1-\sqrt{\tanh(\xi)}\over \sqrt{\xi}}$ is positive and decreasing on $(0,\infty)$, so we have 
	\[
	\int_{0}^{\infty}\cos(x\xi)\cdot{1-\sqrt{\tanh(\xi)}\over \sqrt{\xi}}d\xi~\geq~0\qquad\forall x>0,
	\]
	which yields (\ref{estK}).
	\endproof
	
	Using Theorem \ref{Main} and Lemma \ref{B-K}, we obtain the wave breaking criterion for \eqref{Whithameq} as follows:
	\begin{proposition}\label{Whitham}
		For every $\theta\in \left(0,\tfrac{1}{8}\right)$ and every $\bar{u}\in H^3(\mathbb{R})$ such that
		\bel{cond-u0-whit}
		-\inf_{x\in \mathbb{R}}\bar{u}'(x)~>~{3^{3/4}2^{1/4}C^{1/4}_{GN}\over \pi^{1/4}\theta^{1/2}} \cdot \|\bar{u}'\|^{1/3}_{{\bf L}^{\infty}(\R)}\|\bar{u}'''\|_{{\bf L}^2(\mathbb{R})}^{1/6},
		\eeq
		the Whitham equation (\ref{Whithameq}) exhibits wave breaking at  time $T^*>0$ satisfying (\ref{b-T*}).
	\end{proposition}
	{\bf Proof.} Consider the operator {\bf N}: $H^{1}(\R) \to {\bf L}^{2}(\R)$ defined by
	\[
	{\bf N}[g](x)~\doteq~\left(K * g'\right)(x)~=~\int_{\R} K(y) g'(x - y)d y,
	\]
	which satisfies the assumption {\bf (A1)}. For every $\eta>0$ and $x\in\mathbb{R}$, using (\ref{estK}) and integrating by parts, we estimate 
	\bel{NNN}
	\begin{split}
		| {\bf N}[g](x) |~&\leq~\left|\left( \int_{|y| \leq \eta} + \int_{|y|> \eta}  \right) K(y) g'(x - y)d y\right|~\leq~{\| g' \|_{{\bf L}^{\infty}(\R)}\over \sqrt{2\pi}} \int_{|y| \leq \eta} \frac{1}{\sqrt{|y|}}~d y \\[2mm]
		&\qquad\qquad +|K(\eta)g(x-\eta)-K(-\eta)g(x+\eta)|+\left|\int_{|y|\geq \eta}K'(y) g(x-y)dy\right|\\[2mm]
		&~\leq~{2\sqrt{2}\over \sqrt{\pi}}  \eta^{1/2} \| g' \|_{{\bf L}^{\infty}(\R)}+2
		\eta K(\eta)  \|g'\|_{{\bf L}^{\infty}(\R)} +  \| g \|_{{\bf L}^{\infty}(\R)} \int_{ |y| \geq \eta}  | K'(y) |~d y  \\[2mm]
		&\leq~{3\sqrt{2}\over \sqrt{\pi}}  \eta^{1/2} \| g' \|_{{\bf L}^{\infty}(\R)} +{\sqrt{2}\over \sqrt{\pi}} \eta^{-1/2} \| g \|_{{\bf L}^{\infty}(\R)}.
	\end{split}
	\eeq
	Hence,  ${\bf N}$ satisfies assumption {\bf (A2)} with $\lambda_2=0$ and parameters
	\bel{par21}
	\lambda_1~=~{\sqrt{{2}\over {\pi}}}, \qquad \lambda_3~=~{3{\sqrt{{2}\over {\pi}}}},\qquad \alpha_1~=~\alpha_3~=~\frac{1}{2} \qquad\mathrm{and}\qquad \eta_0~=~+\infty.
	\eeq
	Since $0 < \alpha_{1} = 1/2 < 3/4 = 3 \alpha_{3} / 2$, the hypothesis of Theorem~\ref{Main} is fulfilled in this case. Therefore,  applying Theorem~\ref{Main}  for
	\bel{par22}
	\theta_0~=~\frac{3 \alpha_{3} - 2 \alpha_{1}}{5 \alpha_{1} + 3 \alpha_{3}}~=~\frac{1}{8},\qquad \bar{\alpha}_3~=~\frac{2 \alpha_{1}}{5 \alpha_{1} + 3 \alpha_{3}} = \frac{1}{4},\qquad C_3~=~{3^{9/8}2^{3/8}\over \pi^{3/8}}C_{GN}^{3\over 8},
	\eeq
	we achieve the following wave breaking criterion  for (\ref{Whithameq})
	\[
	-\inf_{x\in \mathbb{R}}\bar{u}'(x)~>~{C_3\over \theta^{3/4}\gamma_{\bar{u}}^{1/2}}\|\bar{u}'''\|_{{\bf L}^2(\mathbb{R})}^{1/4}.
	\]
	Finally, using \eqref{par21}--\eqref{par22} and $\gamma_{\bar{u}}=-\ds\inf_{x\in \R}\bar{u}'(x)/\|\bar{u}'\|_{{\bf L}^\infty(\R)}$, the above criterion can be rewritten as \eqref{cond-u0-whit}. 
	\endproof
	\medskip
	
	\n {\bf 3. Wave breaking for Fornberg-Whitham equation.}
	Consider the Fornberg-Whitham equation
	\bel{FWeq}
	u_{t}(t, x) + \left( \frac{ u^{2}(t, x) }{2} \right)_{x} = (G_{s} * u_{x})(t, x),
	\qquad
	u(0, \cdot) = \bar{u},
	\eeq
	where the kernel $G_{s}$ is the Bessel potential of order $s>0$ whose Fourier multiplier is given by $(1+\xi^2)^{-s/2}$. For every $x\neq 0$, it is computed by 
	\bel{Gs}
	G_s(x)~=~{1\over \sqrt{4\pi}\Gamma(s/2)}\int_{0}^{\infty}t^{(s-3)/2}\cdot e^{-{|x|^2\over 4t}-t}dt,
	\eeq
	with $\Gamma$ being the Gamma function 
	\[
	\Gamma(\beta)~\doteq~\int_{0}^{\infty}t^{\beta-1}e^{-t}dt,\qquad \beta>0.
	\]
	Introduce the function \(\gamma : (0,\infty)\setminus\{1\} \to \mathbb{R}\) defined by  
	\[
	\gamma(s)\doteq \frac{\Gamma\!\left(|s-1|/2\right)}{\sqrt{\pi}\,\Gamma(s/2)}, \qquad s\in (0,\infty)\setminus\{1\}.
	\]
	By direct computation and properties of the Gamma function, the function \(\gamma\) is strictly increasing on \((0,1)\) and strictly decreasing on \((1,\infty)\). Moreover, since $\Gamma(1/2)=\sqrt{\pi}$ and  $\ds\lim_{s\to 0^+}s\Gamma(s)=1$, it holds 
	\bel{p-gamma}
	\lim_{s\to 1}\big(|1-s|\gamma(s)\big)~=~{2\over\pi}, \qquad 
	\lim_{s\to 0^+}\gamma(s)=\lim_{s\to +\infty}\gamma(s)=0.
	\eeq
	\begin{lemma}
		The kernel $G_s:\mathbb{R}\to\mathbb{R}$ is positive, even, and decreasing on $(0,\infty)$. Moreover, for every $x\neq 0$, it holds 
		\bel{b-G-s}
		G_s(x)~\leq~\begin{cases}
			\ds {\gamma(s)\over 2}&\mathrm{if}\qquad s>1,\\[3mm]
			\ds G_1(1)+{|\ln |x||\over \pi}&\mathrm{if}\qquad s=1, 0<|x|\leq 1,\\[3mm]
			\ds {\gamma(s)\over 2^s }\cdot |x|^{s-1}&\mathrm{if}\qquad s\in (0,1).
		\end{cases}
		\eeq
	\end{lemma}
	{\bf Proof.} For every $s>0$, because $G_{s}$ is the Fourier transform of $\left( 1 + \xi^{2} \right)^{-s/2}$, it is well know that $G_{s}$ is positive, even, and monotonic. 
	Hence, it remains to establish bounds for \(|G_s(x)|\) in three cases:
	\begin{itemize}
		\item  If $s>1$ then the first inequality in  (\ref{b-G-s}) follows from 
		\[
		\begin{split}
			G_s(x)&~\leq~{\|\partial_xG_s\|_{{\bf L}^1(\R)}\over 2}~=~{1\over 8\sqrt{\pi}\Gamma(s/2)}\int_{0}^{\infty}t^{(s-5)/2}e^{-t}\left(\int_{\R}|x|e^{-{x^2\over 4t}}dx\right)dt\\[2mm]
			&~=~{1\over 2 \sqrt{\pi}\Gamma(s/2)}\int_{0}^{\infty}t^{(s-3)/2}e^{-t}dt~=~{\Gamma\!\left([s-1]/2\right)\over 2 \sqrt{\pi}\Gamma(s/2)}.
		\end{split}
		\]
		\item  In the case  $s=1$, for every $x> 0$ we compute 
		\[
		\begin{split}
			\big|\partial_xG_1(x)\big|&~=~\left|{d\over dx}\left({1\over 2\pi}\int_0^{\infty}t^{-1}e^{-t}e^{-x^2\over 4t}dt\right)\right|~=~\left|{x\over 4\pi}\int_{0}^{\infty}t^{-2}e^{-t}e^{{-x^2\over 4t}}dt\right|\\[3mm]
			&~\leq~{|x|\over 4\pi}\cdot \int_{0}^{\infty}t^{-2}e^{{-x^2\over 4t}}dt~=~{1\over \pi |x|},
		\end{split}
		\]
		which implies the second inequality in (\ref{b-G-s}).
		\item Finally, if $0<s<1$, then for every $x>0$, it holds
		\[
		0~\leq~ G_s(x)~\leq~{1\over 2\sqrt{\pi}\Gamma(s/2)}\int_{0}^{\infty}t^{(s-3)/2}\cdot e^{-{x^2\over 4t}}dt.
		\]
		By a change of variable $u=\ds {x^2\over 4t}$, we have 
		\[
		\int_{0}^{\infty}t^{(s-3)/2}\cdot e^{-{x^2\over 4t}}dt~=~{2^{1-s}\over x^{1-s}}\cdot \int_{0}^{\infty}u^{(1-s)/2-1}e^{-u}du~=~{{2^{1-s}\Gamma([1-s]/2)}\over x^{1-s}},
		\]
		which yields the last inequality in  (\ref{b-G-s}). 
	\end{itemize}
	The proof is complete.
	\endproof
	\medskip

	Combining Theorem~\ref{Main} with the above lemma, we derive a wave-breaking criterion for \eqref{Whithameq}. For each $s \in (0,1)$, define
	\bel{Cond1}
	\beta_1(s)~=~[3 \gamma(s)]^{\frac{1}{2}}
	\left[
	\frac{2+2s}{s}\cdot C_{GN}
	\right]^{\frac{1- s}{2}},
	\qquad 
	\beta_2(s)~=~[3 \gamma(s)]^{\frac{1}{2}}
	\left[
	\frac{{2}^{2s-1}}{{(2s-1)}^{1-s}}
	\right]^{\frac{1}{2}}.
	\eeq
	Our main result is as follows:
	\begin{proposition}\label{FW} 
		For  the Fornberg-Whitham equation (\ref{FWeq}), assume that one of the following holds:
		\begin{enumerate}
			\item [(i).] For   $s \in (2/5, 1)$ and  $\theta\in \left(0,\frac{5s-2}{5-2s} \right)$, the initial data  $\bar{u}\in H^3(\mathbb{R})$ satisfies 
			\bel{cond-u0-FW1}
			-\inf_{x\in \R}\bar{u}'(x)~>~{\beta_1(s)\over \theta^{1/2}} \cdot \| \bar{u}' \|^{ \frac{1}{6}+\frac{s}{3} }_{{\bf L}^{\infty}(\R)} \| \bar{u}''' \|_{{\bf L}^2(\mathbb{R})}^{ \frac{1}{3} - \frac{s}{3} }.
			\eeq
			\item [(ii).] For $s \in (2/3,1)$ and $\theta\in \left(0,{3s-2\over 3-2s} \right)$, the initial data  $\bar{u}\in H^2(\mathbb{R})$ satisfies
			\bel{cond-u0-FW2}
			-\inf_{x\in \R}\bar{u}'(x)~>~{\beta_2(s)\over \theta^{1/2} }\cdot {\| \bar{u}' \|^{ s-{1\over 2} }_{{\bf L}^{\infty}(\R)}}  \|\bar{u}''\|^{1-s}_{{\bf L}^2(\R)}.
			\eeq
			\item [(iii).] For $s = 1$, $\tau \in (2/3, 1)$, and $\theta\in \left(0,{3\tau-2\over 3-2\tau}\right)$, the initial data  $\bar{u}\in H^2(\mathbb{R})$ satisfies
			\bel{cond-u0-FW3}
			-\inf_{x\in \R}\bar{u}'(x)~>~\max\left\{\sqrt{12\over \pi (1-\tau)\theta}\cdot {\|\bar{u}'\|_{{\bf L}^\infty(\R)}^{1\over 2}}, \sqrt{12\over \pi (1-\tau)\theta} \cdot {\| \bar{u}' \|^{ \tau-\frac{1}{2} }_{{\bf L}^{\infty}(\R)}} \|\bar{u}''\|^{1 - \tau}_{{\bf L}^2(\R)}\right\}.
			\eeq
			\item [(iv).] For $s\in (1,\infty)$ and  $\theta\in (0,1)$, the initial data $\bar{u}\in H^2(\mathbb{R})$ satisfies 
			\bel{cond-u0-FW4}
			-\inf_{x\in \R}\bar{u}'(x)~>~\sqrt{3 \gamma(s)\over \theta}\cdot \|\bar{u}'\|_{{\bf L}^\infty(\R)}^{1\over 2}.
			\eeq
			
		\end{enumerate}
		Then the wave breaking time $T^*$ of the Fornberg-Whitham equation (\ref{FWeq}) is bounded by (\ref{b-T*}).
	\end{proposition}
	{\bf Proof.} For every $s>0$, the operator {\bf N}: $H^{2}(\R) \to H^{1}(\R)$ defined by
	\[
	{\bf N}[g](x)~\doteq~(G_{s} * g')(x)~=~\int_{\R} G_{s}(y) g'(x - y)d y
	\]
	satisfies the assumption {\bf (A1)}. We verify the assumption {\bf (A2)} in four cases:
	\medskip
	
	{\bf 1.} If $s \in (2/5, 1)$, then following the same estimate as in (\ref{NNN}) together with (\ref{b-G-s}), for every $\eta >0$ and $x>0$, we estimate 
	\[
	\begin{split}
		|{\bf N}[g](x)|&~\leq~\left( \int_{|y| \leq \eta} + \int_{|y| > \eta} \right) G_{s}(y) g'(x - y)d y~\leq~ \| g' \|_{{\bf L}^{\infty}(\R)} \int_{|y| \leq \eta} |G_{s}(y)|dy\\[2mm]
		&\qquad\qquad+\left| G_{s}(\eta) g(x - \eta)-G_{s}(- \eta) g(x + \eta) \right| + \| g \|_{{\bf L}^{\infty}(\R)} \int_{|y|\geq \eta} |\partial_y G_{s}(y)|d y\\[2mm]
		&~\leq~\|g' \|_{{\bf L}^{\infty}(\R)}\left( \int_{|y| \leq \eta} |G_{s}(y)|dy+2 \eta |G_s(\eta)|\right)+2 \|g\|_{{\bf L}^{\infty}(\R)}\cdot |G_{s}(\eta)|\\[2mm]
		&~\leq~{2^{1-s}(1+s)\gamma(s)\over s}\cdot \eta^{s} \| g' \|_{{\bf L}^{\infty}(\R)} +2^{1-s}\gamma(s)\cdot \eta^{s-1} \| g \|_{{\bf L}^{\infty}(\R)}.
	\end{split} 
	\]
	In particular,  ${\bf N}$ satisfies assumption {\bf (A2)} with $\lambda_2=0$ and parameters
	\bel{FWA2}
	\left( \lambda_{1}, \lambda_{3} \right)~=~ 2^{1-s}\left(\gamma(s),{(1+s)\gamma(s)\over s} \right),\qquad (\alpha_{1}, \alpha_{3})~=~(1 - s, s) \qquad\mathrm{and}\qquad \eta_0~=+ \infty .
	\eeq
	Since $s\in (2/5,1)$, it holds
	\[
	0 ~<~ \alpha_{1} ~=~ 1-s ~<~ \ds{3s\over 2} ~=~ \ds{3 \alpha_{3} \over 2},
	\]
	and the hypotheses of Theorem~\ref{Main} are satisfied. Hence, applying Theorem~\ref{Main} with
	\[
	\begin{cases}
		\theta_0~=~\ds{3\alpha_3-2\alpha_1\over 5\alpha_1+3\alpha_3 }~=~\frac{5s-2}{5-2s},\qquad \bar{\alpha}_3~=~{2\alpha_1\over 5\alpha_1+3\alpha_3}~=~ \frac{2-2s}{5-2s},\\[4mm]
		\ds C_2~=~0,\qquad C_3~=~3^{1-\bar{\alpha}_3}\lambda_1^{{2-5\bar{\alpha}_3\over 2}}\big(C_{GN}\lambda_3\big)^{^{{3\bar{\alpha}_3\over 2}}}~=~\left[3C_{GN}^{1-s}\left(\frac{2+2s}{s}\right)^{1-s}\gamma(s) \right]^{\frac{3}{5-2s}},
	\end{cases}
	\]
	and recalling (\ref{Cond1}), we achieve the wave breaking criterion as in \eqref{cond-u0-FW1}.
	\medskip
	
	{\bf  2.} If $s \in (2/3,1)$ then for every $\eta >0$ and $x>0$, we estimate 
	\[
	\begin{split}
		|{\bf N}[g](x)|&~\leq~\left( \int_{|y| \leq \eta} + \int_{|y| > \eta} \right) G_{s}(y) g'(x - y)d y~\leq~ \| g' \|_{{\bf L}^{2}(\R)} \left(\int_{|y| \leq \eta} |G_s(y)|^2dy\right)^{1/2}\\[2mm]
		&\qquad+\left| G_{s}(\eta) g(x - \eta)\right|+\left|G_{s}(- \eta) g(x + \eta) \right| + \| g \|_{{\bf L}^{\infty}(\R)}\cdot \int_{|y|\geq \eta} |\partial_yG_s(y)|d y\\[2mm]
		&~\leq~ {2^{-s+(1/2)}\gamma(s)\over \sqrt{2s-1}}\cdot \eta^{s-(1/2)} \| g' \|_{{\bf L}^{2}(\R)} +2^{2-s}\gamma(s)\cdot \eta^{s-1} \| g \|_{{\bf L}^{\infty}(\R)},
	\end{split} 
	\]
	which implies that ${\bf N}$ satisfies assumption {\bf (A2)} with $\lambda_3=0$ and parameters
	\bel{FWA3}
	\left( \lambda_{1}, \lambda_{2} \right)~=~\left(2^{2-s}\gamma(s),{2^{-s + (1/2)}\gamma(s)\over \sqrt{2s-1}}\right),\quad (\alpha_{1}, \alpha_{2})~=~(1 - s, s-(1/2)) \quad\mathrm{and}\quad \eta_0~=+\infty .
	\eeq
	Since $s\in (2/3,1)$, it holds 
	\[
	0~<~ \alpha_{1} ~=~ 1-s~ <~ 2s-1 ~=~2\alpha_2,
	\]
	and  the hypothesis of Theorem~\ref{Main} is fulfilled in this case.  Hence,  applying Theorem~\ref{Main}  for
	\[
	\begin{cases}
		\theta_0~=~\ds{2\alpha_2-\alpha_1\over 4\alpha_1+2\alpha_2}~=~{3s-2\over 3-2s},\qquad \bar{\alpha}_2~=~{\alpha_1\over 2\alpha_1+\alpha_2}~=~{2-2s\over 3-2s },\\[4mm]
		C_2~=~ 3^{1-\bar{\alpha}_2}\lambda^{1-2\bar{\alpha}_2}_1\lambda_2^{\bar{\alpha}_2}~=~\ds  3^{\frac{1}{3-2s}}
		\,2^{\frac{2s - 1}{3 - 2s} }
		\,\gamma(s)^{\frac{1}{3 - 2s}}
		\,(2s-1)^{-\frac{1-s}{3-2s}},\qquad C_3~=~0,
	\end{cases}
	\]
	and recalling (\ref{Cond1}), we achieve the wave breaking criterion as in (\ref{cond-u0-FW2}).
	\medskip

	{\bf 3.} If $s=1$, then with the same estimate as in Case 2, for every $\eta\in (0,1]$ we have 
	\[
	\big\|{\bf N}[g]\big\|_{{\bf L}^{\infty}(\R)}~\leq~4 \| g \|_{{\bf L}^{\infty}(\R)} |G_1(\eta)|+ \| g' \|_{{\bf L}^{2}(\R)}\left(\int_{|y| \leq \eta} |G_{1}(y)|^2dy\right)^{1/2}.
	\]
	By the second bound in (\ref{b-G-s}) and $\ds G_1(1)< \frac{1}{\pi}$, we directly compute that  
	\[
	\begin{split}
		\big\|{\bf N}[g]\big\|_{{\bf L}^{\infty}(\R)}&~\leq~{4(1+|\ln \eta|)\over \pi}\cdot \| g \|_{{\bf L}^{\infty}(\R)}+{4(1+|\ln\eta|)\over \pi}\eta^{1/2} \| g' \|_{{\bf L}^{2}(\R)}\\[2mm]
	\end{split}
	\]
	In particular, for every fixed $\tau\in (2/3,1)$, using the fact that
	\[
	1+|\ln\eta|~\leq~{\eta^{\tau-1}\over 1-\tau},\qquad\eta\in (0,1],
	\] 
	we derive 
	\[
	\big\|{\bf N}[g]\big\|_{{\bf L}^{\infty}(\R)}~\leq~{4\over \pi (1-\tau) }\eta^{\tau-1}\| g \|_{{\bf L}^{\infty}(\R)}+{4\over \pi (1-\tau) }\eta^{(\tau-1/2)}\| g' \|_{{\bf L}^{2}(\R)},\qquad \eta\in (0,1].
	\]
	Hence, ${\bf N}$ satisfies assumption {\bf (A2)} with $\lambda_3=0$ and parameters
	\bel{FWA4}
	\left( \lambda_{1}, \lambda_{2} \right)~=~\left({4\over \pi (1-\tau) },{4\over \pi (1-\tau) }\right),\quad (\alpha_{1}, \alpha_{2})~=~(1-\tau, \tau-1/2) \quad\mathrm{and}\quad \eta_0~=1 .
	\eeq
	Since $\tau\in (2/3,1)$, it holds
	\[
	0~<~ \alpha_{1} ~=~ 1-\tau~ <~ 2\tau-1 ~=~2\alpha_2,
	\]
	and  the hypothesis of Theorem~\ref{Main} is fulfilled. Hence,  applying Theorem~\ref{Main}  for
	\[
	\begin{cases}
		\ds \theta_0~=~\ds{2\alpha_2-\alpha_1\over 4\alpha_1+2\alpha_2}~=~{3\tau-2\over 3-2\tau},\qquad\bar{\alpha}_2~=~{\alpha_1\over 2\alpha_1+\alpha_2}~=~{2-2\tau\over 3-2\tau},\qquad \eta_0~=~1,\\[4mm]
		\ds C_2~=~3^{1-\bar{\alpha}_2}\lambda^{1-2\bar{\alpha}_2}_1\lambda_2^{\bar{\alpha}_2}~=~\left[{12\over\pi(1-\tau) }\right]^{1\over 3-2\tau},\qquad C_3~=~0,
	\end{cases}
	\]
	we get (\ref{cond-u0-FW3}). 
	\medskip
	
	{\bf 4.}  Finally, if $s>1$, then by the first inequality in (\ref{b-G-s}) we have 
	\[
	\begin{split}
		\big\|{\bf N}[g]\big\|_{{\bf L}^1(\R)}&~=~\big\|\partial_xG_{s} * g\big\|_{{\bf L}^1(\R)}~\leq~\|\partial_xG_{s}\|_{{\bf L}^1(\R)}\cdot  \|g\|_{{\bf L}^{\infty}(\R)}\\[2mm]
		&~\leq~{\gamma(s)}\cdot \|g\|_{{\bf L}^{\infty}(\R)}\quad\forall g\in {\bf L}^{\infty}(\R).
	\end{split}
	\]
	Hence, by Proposition \ref{L1-ker}, we achieve (\ref{cond-u0-FW4}).
	\endproof	
	\begin{remark} From (\ref{p-gamma}) and (\ref{Cond1}), it holds
		\[
		\lim_{s\to 1^-} \Big(\sqrt{1-s}\cdot \beta_2(s)\Big)~=~{\sqrt{12\over \pi}}.
		\]
		This reveals a sharp connection between the wave-breaking criterion for the subcritical case $s \in (2/3, 1)$ in (\ref{cond-u0-FW2}) and the critical case $s = 1$ in (\ref{cond-u0-FW3}). On the other hand, a different approximation procedure for (\ref{cond-u0-FW3}) could be used to yield a similar connection relating the supercritical case $s \in (1, \infty)$ in (\ref{cond-u0-FW4}) to the critical case $s = 1$.
	\end{remark}
	
	{\bf Acknowledgements.} This research was partially supported by NSF-DMS 2154201 (K.T. Nguyen).

\end{document}